\documentclass[11pt]{article}

\usepackage[utf8]{inputenc}
\usepackage[T1]{fontenc}
\usepackage[english]{babel}

\usepackage{xcolor}
\usepackage{hyperref}
\hypersetup{%
  colorlinks=false,
  linkbordercolor=cyan
}

\usepackage{mathtext}
\usepackage{amsmath,amsfonts,amssymb,amsthm, mathtools} % математика
\usepackage{icomma} % умная запятая
\usepackage{import} %  импортирование

\newtheorem{proposition}{Proposition}

\newtheorem{theorem}{Theorem}
\theoremstyle{definition}
\newtheorem{definition}{Definition}

\newtheorem{example}{Example}
\newtheorem{lemma}{Lemma}
\theoremstyle{remark}
\newtheorem{remark}{Remark}

\usepackage[left=2.5cm,right=2.2cm,top=2.5cm,bottom=3.5cm]{geometry}

\DeclareMathOperator{\ord}{ord}

\DeclareMathOperator{\WP}{WP}

\DeclareMathOperator{\supp}{supp}

\newcommand{\eqdef}{\stackrel{\mathrm{def}}{=}}

\usepackage{tikz}
\usepackage{tikz-cd} % коммутативные диаграммы
\usetikzlibrary{cd}
\usepackage{pgfplots}
\pgfplotsset{width=6cm,compat=newest}

\DeclareMathSymbol{\mlqq}{\mathord}{operators}{``}
\DeclareMathSymbol{\mrqq}{\mathord}{operators}{`'}

\title{Tropical Weil's reciprocity law and Weil's pairing}

\author{Nikita Kalinin, Matthew Magin} 

\date{\today}
     
\begin{document}

\maketitle

\begin{abstract}
       The Weil reciprocity law asserts that given two meromorphic functions $f, g$ on a compact complex curve, the product of the values of $f$ over the roots and poles of $g$ is equal to the product of the values of $g$ over the roots and poles of $f$. We state and prove a tropical version of this reciprocity; the tropical ideas lead to yet another transparent ``combinatorial'' proof of the classical Weil reciprocity law. Then, we construct a tropical Weil pairing on the set of divisors of degree zero. 
\end{abstract}
     
     \tableofcontents

 \section{Weil's reciprocity law}

   The first known formulation of Weil's reciprocity law is contained in Andr\'e Weil's letter to Emil Artin on 10 July 1942, \cite[p. 291]{weil2009oeuvres}. In this letter, Weil explains his definition of a pairing \cite{weil40} on the classes of divisors of finite order. This reciprocity law was used as an ingredient in that construction; Weil did not even have a separate statement for it in \cite{weil40}.

  Let us formulate this reciprocity law over $\mathbb C$. Let $f$ be a meromorphic function on a Riemann surface $S$. Then, in a small neighborhood of a point $p$, choose a local parameter $z$, then the function $f$ can be expanded as 
  \[
    f(z) = a_0 z^{k} + a_1 z^{k + 1} + \ldots, a_0\ne 0.
   \]
    The smallest degree $k$ in this expansion does not depend on the choice of a local parameter. It is called \emph{the order of the function $f$ at the point $p$} and denoted by $\ord_{p}{f}=k$.

  \begin{theorem}[\cite{{weil2009oeuvres}}]\label{complex_weil_reciprocity} 
    Let $f, g$ be non-zero meromorphic functions on the compact Riemann surface $S$ with no common zeros or poles. Then 
    \[
      \prod_{p \in S} f(p)^{\ord_{p}{g}} = \prod_{p \in S} g(p)^{\ord_{p}{f}}
    \]

  \end{theorem}

  The products in this theorem are well-defined since a meromorphic function can have only a finite number of zeros and poles on a compact Riemann surface. Weil's reciprocity law may be thought of a generalization of Vieta's theorem because in the case $S=\mathbb C P^1, f(z)=z, g(z) = \sum_{k=0}^n a_kz^k$ the product $\prod_{p \in \mathbb C} f(p)^{\ord_{p}{g}}$ is equal to $a_0/a_n$.
 
   It is not difficult to extend this theorem for the case of any two non-zero meromorphic functions on $S$, even for $f$ and $g$ having common zeroes or poles. For two meromorphic functions $f, g$ on a compact Riemann surface $S$ we can define \emph{Weil's symbol} $[f, g]_{p}$ as
  \[
    [f, g]_{p} = (-1)^{nm} \cdot \frac{a_n^{m}}{b_{m}^n}, 
  \]
  where $f(z) = a_nz^n  + \ldots, \ g(z) = b_mz^m + \ldots, a_n\ne 0, b_m\ne 0$ are expansions of $f$ and $g$ in a small neighbourhood of the point $p \in S$. Again, $[f, g]_{p}$ does not depend on a choice of a local parameter $z$.  Then, Weil's reciprocity law can be rewritten as a product formula
  \[
    \prod_{p \in S} [f, g]_{p} = 1. 
  \]

 {\bf History.} Weil's proof of this reciprocity law is reminiscent of the famous Menelaus' and Carnot's theorems. Menelaus' (70--140 CE) theorem states that given a triangle $ABC$ and a line $l$ crossing lines $AB,BC,AC$ in points $C',A',B'$ one has $$\frac{|AC'|}{|C'B|}\cdot\frac{|BA'|}{|A'C|}\cdot\frac{|CB'|}{|B'A|}=1.$$

Lazare Carnot's (1753--1823) theorem states that given a triangle $ABC$ and a conic crossing $AB,BC,AC$ in points $C_1',C_2',A_1',A_2',B_1',B_2'$ one has

$$\frac{|AC_1'|\cdot |AC_2'|}{|C_1'B|\cdot |C_2'B|}\cdot\frac{|BA_1'|\cdot|BA_2'|}{|A_1'C|\cdot |A_2
C|}\cdot\frac{|CB_1'|\cdot |CB_2'|}{|B_1'A|\cdot|B_2'A|}=1.$$

A similar statement can be made for any degree $d$ curve crossing each of the lines $AB, BC, AC$ in exactly $d$ points. The proof amounts to parameterizing each of the lines $AB, BC, AC$ in such a way that $|AC_1'|\cdot |AC_2'|\cdot\dots |AC_d'|$ becomes the product of the roots of a polynomial in one variable (the restriction of the equation of the curve on the line $AB$) and thus Vieta's theorem can be used several times, then a simple comparison of results concludes the proof. See more details about Menelaus' and Carnot's theorems in \cite{murchadha2012menelaus}.

Weil's proof of its reciprocity law follows the same line of arguments, for he considered a curve $C$ in $\mathbb C P^1\times \mathbb C P^1$ given by the image of $S\ni z \mapsto (f(z),g(z))$ and intersected $C$ with four coordinate lines of $\mathbb C P^1\times \mathbb C P^1$. Interestingly, one can state the converse of this reciprocity, which gives necessary and sufficient conditions for the existence of an algebraic curve in a toric surface with given intersection with boundary divisors \cite{weil_kh}.

    This reciprocity law turned out to be fundamental and has many generalizations in number theory; see \cite{serre2012algebraic}, Chapter 3, and further development in \cite{coleman1989reciprocity}. It may be proven  analytically, e.g., \cite[pp.  242–243]{griffiths2014principles}, \cite{previato1991another}, or topologically, e.g., \cite{weil1}.  Belinson mentioned that a generalization of this reciprocity law was the initial point of his famous work \cite{beilinson1985higher} on Belinson--Deligne cohomology theory; see also Deligne's work \cite{deligne1991symbole}.

    Weil's pairing is used in cryptography; see \cite{miller2004weil}. There are generalizations of this pairing for graphs and tropical curves, under the name of the monodromy pairing or Grotendieck's pairing, see \cite{lorenzini2000arithmetical}, \cite{MR1478029},  \cite{shokrieh2010monodromy}, \cite{bosch2002grothendieck}.

  \section{Tropical Weil's reciprocity law}
  
   Tropical curves are topological graphs with a metric on the edges.  Tropical curves were defined by Mikhalkin in \cite{mikh1} and appeared there as degenerations of algebraic curves, similar in spirit to Viro's patchworking technique \cite{viro2006patchworking}. In that case, tropical curves may have edges of infinite length, certain enhancements such as genus function, they may be immersed into plane or space, etc.; see \cite{mikh2, BIMS} for thorough expositions. Similarities between metric graphs and algebraic curves were also observed in other contexts; see \cite{MR1019714, MR1478029}. For the purpose of this article, we assume the following simplified definition.
    \begin{definition}\label{abstract_tropical_curve}
    A \emph{tropical curve} $\Gamma$ is a connected metric graph with a finite number of vertices and edges.
  \end{definition}

  \begin{definition}
      A \emph{tropical meromorphic function on a tropical curve $\Gamma$} is a piece-wise linear (with integer slope) continuous function on $\Gamma$.
  \end{definition}

\begin{definition}\label{rem_order}

The \emph{order} $\ord_{p}{f}$ of a tropical meromorphic function $f$ at a point $p$ on a tropical curve $\Gamma$ is the sum of outgoing slopes of $f$ around $p$. If $p$ is a one-valent vertex, then there is only one direction from $p$, and the (integer) slope of $f$ in this direction is the order of $f$ at $p$.  If $p$ belongs to the interior of an edge, then we take the sum of two slopes of $f$ (there are two directions emanating from $p$).  If  $p$ is a vertex of $\Gamma$ and $e_{1}, \ldots, e_{k}$ be the edges of $\Gamma$ adjacent to $p$ then 
     \begin{equation}
      \ord_{p}{f} = \sum_{j = 1}^{k} \ord_{p}\left(f\vert_{e_j}\right).
     \end{equation}
\end{definition}

We say that the tropical meromorphic function $f$ has a \emph{pole of multiplicity $m$} $x \in \Gamma$ if $\ord_{x}{f} = -m$ and \emph{zero of multiplicity $k$ in $x \in \Gamma$} if $\ord_{x}{f} = k$.  The divisor $(f)$ of a tropical meromorphic function is the formal sum $\sum_{p\in\Gamma} p\cdot \ord_pf$. 

  \begin{example}
    Consider a tropical meromorphic function $f(x) = \max(x, 3) + \max(x, 2) - \max(x, 1)$ on $[0,7]$. We have $\ord_{2}{f} = \ord_{3}{f} = 1, \ \ord_{1}{f} = \ord_{7}{f} = -1$ and $\ord_{x}{f} = 0$ for all other $x \in [0, 7]$.
    \begin{center}
      \begin{tikzpicture}
      \begin{axis}[
          axis lines = center,
          xlabel = \(x\),
          ylabel = {\(f(x)\)},
          xmin=0, xmax=7,
          ymin=0, ymax=7,
          xtick={0, 1, 2, 3,7},
          ytick={0, 3, 4,7},
      ]
      \addplot [
          domain=0:10, 
          samples=1000, 
          color=red,
      ]
      {max(x, 3) + max(x, 2) - max(x, 1)};
      
    \end{axis}
    \end{tikzpicture}
  \end{center}
    
  \end{example}

  In tropical geometry, multiplication is replaced with addition, so a tropical analog of Theorem~\ref{complex_weil_reciprocity} is the following theorem.

  \begin{theorem}[Tropical Weil's reciprocity law]\label{Tropical_Weil_Reciprocity}
    Let $\Gamma$ be a tropical curve and $f, g$ be two tropical meromorphic functions on $\Gamma$. Then 
  % \begin{equation}
  %   \bigodot_{x \in \Gamma} f(x)^{\ord_{x}{g}} = \bigodot_{x \in \Gamma} g(x)^{\ord_{x}{f}}.  
    %\end{equation}
      
    %In standard arithmetic on $\R$ this statement looks like 
    \begin{equation}
      \sum_{x \in \Gamma} f(x) \cdot \ord_{x}{g}  = \sum_{x \in \Gamma} g(x) \cdot \ord_{x}{f} .
    \end{equation}
  \end{theorem}
  \begin{proof}
    First, we prove the tropical reciprocity law for $\Gamma = e$ where $e$ is a metric segment $[0, L]$. Denote by $$\{ x_1, \ldots, x_n \}, \ x_1 \le x_2 \le \ldots, \le x_n$$ the set of poles and zeroes of $f$ and $g$ (i.e. the union of supports of their divisors). 

    Let $k_i$  be a slope of $f$ on a segment $[x_i, x_{i + 1}]$, $\ell_i$ be a slope of $g$ on a segment $[x_i, x_{i + 1}]$. Then the formula $\sum_{x \in e} f(x) \cdot \ord_{x}{g}  = \sum_{x \in e} g(x) \cdot \ord_{x}{f}$ is equivalent to 
    \begin{equation}
      \ell_1 f(x_1) + \sum_{i = 2}^{n - 1}(\ell_{i} - \ell_{i - 1})f(x_i)  - \ell_{n - 1} f(x_n) = k_1 g(x_1) + \sum_{i = 2}^{n} (k_{i} - k_{i - 1}) g(x_i) - k_{n - 1} g(x_n). \label{eq_1}
    \end{equation}
    
    The left side and the right side, correspondingly, are $$\sum_{i = 1}^{n - 1}\ell_{i} (f(x_{i}) - f(x_{i + 1})) \text{ \  and \ } \sum_{i = 1}^{n - 1} k_{i}(g(x_i) - g(x_{i + 1})).$$

     As $f$ and $g$ are piece-wise linear, we have $k_i = \frac{f(x_{i + 1}) - f(x_i)}{x_{i + 1} - x_i}, \ \ell_i = \frac{g(x_{i + 1}) - g(x_i)}{x_{i + 1} - x_{i}}$, so  
     \begin{equation}
      \sum_{i = 1}^{n - 1}\ell_{i} (f(x_{i}) - f(x_{i + 1})) = \sum_{i = 1}^{n - 1} \frac{g(x_{i + 1}) - g(x_{i})}{x_{i + 1} - x_{i}} (f(x_{i}) - f(x_{i + 1}))
     \end{equation}
     and  
     \begin{equation}
      \sum_{i = 1}^{n - 1} k_{i}(g(x_i) - g(x_{i + 1})) = \sum_{i = 1}^{n - 1} (g(x_i) - g(x_{i + 1})) \frac{f(x_{i + 1}) - f(x_i)}{x_{i + 1} - x_{i}}, 
     \end{equation}

     from what we obtain that the equality~\eqref{eq_1} holds.

     Let $\Gamma$ be an arbitrary tropical curve. Denote the set of edges of $\Gamma$ by $E(\Gamma) = \{ e_1, \ldots, e_m \}$.

%     \begin{figure}[htbp]
%         \begin{tikzpicture}[scale = 0.8]
%          \node at (0,0){\includegraphics[scale = 1]{pictures/pic_0.pdf}};
%          \node[right] at (-0.2, -0.6){\small \( u \)};
%          \node[right] at (-0.15, 0.35){\small \( v_1 \)};
%          \node[right] at (0.8, -1.2){\small \( v_2 \)};
%          \node[right] at (-1.65, -1.15){\small \( v_3 \)};
%          %\input{rulers.tex}
%        \end{tikzpicture}
%        \centering
%        \caption{Additive property of the order.}
%        \label{fig:0}
%        \end{figure}

        Using the formula in Definition~\ref{rem_order} we obtain
     
     \begin{equation}
      \sum_{x \in \Gamma} g(x) \cdot \ord_{x}{f} = \sum_{j = 1}^{m} \sum_{x \in e_{j}} g(x) \cdot \ord(f\vert_{e_j}) = \sum_{j = 1}^{m}\sum_{x \in e_{j}}  f(x) \cdot \ord_{x}(g\vert_{e_{j}}) = \sum_{x \in \Gamma} f(x) \cdot \ord_{x}{g},
     \end{equation}

     as for every edge $e_j$ we have $\sum_{x \in e_{j}} g(x) \cdot \ord(f\vert_{e_j}) = \sum_{x \in e_{j}}  f(x) \cdot \ord_{x}(g\vert_{e_{j}})$. 

  \end{proof}

  \begin{definition}\label{Weil_symbol}
    Let $f$ and $g$ be two tropical meromorphic functions on a tropical curve $\Gamma$. \emph{Tropical Weil's symbol } of $f, g$ at the point $p \in \Gamma$ is 
    \[
      [f, g]_{p} \eqdef \ord_{p}{g} \cdot f(p)  - \ord_{p}{f} \cdot g(p). 
    \]
  \end{definition}

  As in the complex algebraic geometry case, we can rewrite Theorem~\ref{Tropical_Weil_Reciprocity} like a product formula. 
    
  \begin{theorem}[Tropical Weil's reciprocity law]
    Let $\Gamma$ be a compact tropical curve and $f, g$ be two meromorphic functions on $\Gamma$. Then we have 
    \begin{equation}
      \sum_{x \in \Gamma}[f, g]_{x} = 0.
    \end{equation}
  \end{theorem}
    
  Classical Weil's symbol depends on functions $f$ and $g$ skew-symmetrically and is multiplicative in each argument (see for example~\cite[p. 86]{weil1}). Tropical Weil's symbol has similar properties. One can prove the following theorem by direct computation.

  \begin{proposition} 
    Tropical Weils symbol satisfies the following properties:
    \begin{enumerate}
      \item The Tropical Weil symbol depends on functions f, g skew-symmetrically (in tropical sense): $[f, g]_{x} = - [g, f]_{x}$.

      \item The Tropical Weil symbol is additive  in each argument, i.e., for any triplet of non-zero tropical meromorphic functions $f, g, \varphi$ and a point $x \in \Gamma$: 
      \begin{equation}
        [f + \varphi, g]_{x} = [f, g]_{x} + [\varphi, g]_{x}, \quad [f, \varphi + g]_{x} = [f, \varphi]_{x} + [f, g]_{x}. 
      \end{equation}
    \end{enumerate}
  \end{proposition}
%   \begin{proof}
%     Direct computation.
%   \end{proof}

\begin{remark}
It is not difficult to extend the definition of the tropical Weil symbol to the case of tropical curves with edges of infinite lengths. 
\end{remark}

One can prove the tropical Weil reciprocity law with tropical modifications and tropical momentum map; see \cite{mythesis}, Chapter 3.  The tropical momentum map is a tropical version of Menelaus' and Carnot's theorems; see \cite{Mikhalkin:2015kq}. This way of proving the tropical Weil reciprocity law boils down to mapping a tropical curve to a planar tropical plane (using two meromorphic functions on it) and then expressing $\sum_{x \in \Gamma} f(x) \cdot \ord_{x}{g} $ via the coefficients of the equation of the planar tropical curve. This approach is a tropical version of Khovanskii's approach \cite{weil_kh} to the intersection of planar curves with divisors at infinity. It is an interesting problem how to detropicalize (i.e. find a complex version) of a generalized tropical momentum map (as it is defined in \cite{mythesis} Chapter 3, p.80) to get multidimensional version of Parshin symbols, as in \cite{Khovanskii1}, but for higher codimension.

Note that we did not use in the proof that the slopes of our functions are integers. Consider continuous piece-wise linear functions on a tropical curve $\Gamma$ with no restrictions on slopes. Call such functions {\it tropical quasi-meromorphic functions}. The notion of the divisor of such a function is defined as for the tropical meromorphic function. Thus, we obtain the following theorem.

\begin{theorem}[Tropical Weil's reciprocity law for tropical harmonic functions]\label{Tropical_Weil_Reciprocity_harmonic}
    Let $\Gamma$ be a tropical curve and $f, g$ be two tropical quasi-meromorphic functions on $\Gamma$. Then 
  % \begin{equation}
  %   \bigodot_{x \in \Gamma} f(x)^{\ord_{x}{g}} = \bigodot_{x \in \Gamma} g(x)^{\ord_{x}{f}}.  
    %\end{equation}
      
    %In standard arithmetic on $\R$ this statement looks like 
    \begin{equation}
      \sum_{x \in \Gamma} f(x) \cdot \ord_{x}{g}  = \sum_{x \in \Gamma} g(x) \cdot \ord_{x}{f} .
    \end{equation}
  \end{theorem}

   \subsection{Combinatorial proof of Weil's reciprocity law}\label{alt_proof_of_weil_reciprocity}

    In this section we give a proof of the classical Weil's reciprocity law (Theorem~\ref{complex_weil_reciprocity}), based on tropical ideas. One can briefly summarize the idea of the proof of Theorem~\ref{Tropical_Weil_Reciprocity}  as follows. Cut the tropical curve into edges, prove the statement for the restriction of the functions on each edge, and check that the condition of the theorem preserves gluing, which is equivalent to the Definition~\ref{rem_order}.

    A tropical curve can be seen as a limit of a family of complex algebraic curves. We can imagine that during this degeneration, the cylinders of the Riemann surface $S$ shrink to the edges of the tropical curve $\Gamma$ (Figure~\ref{fig:1}).  

    \begin{figure}[htbp]
         \begin{tikzpicture}[scale = 1]
          \node at (0,0){\includegraphics[scale = 1]{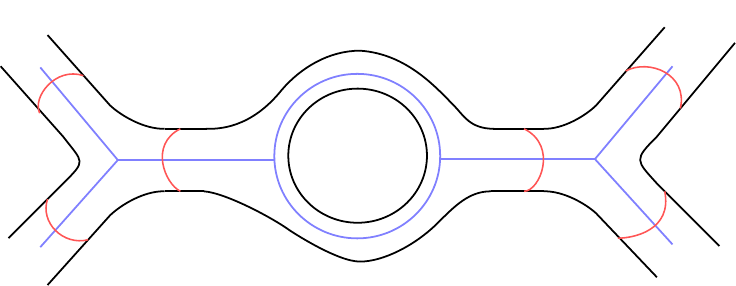} };
          \node[right] at (1.2, 0){\small \( \Gamma \)};
        \end{tikzpicture}
        \centering
        \caption{A Riemann surface $S$ shrinks to the tropical curve $\Gamma$.}
        \label{fig:1}
        \centering
        \end{figure}

    To simplify the notation in further computations, assume the following notation. 

    \begin{definition} 
        On a compact surface $S$ (possibly with a non-empty boundary) for two nonzero meromorphic functions $f, g$ with disjoint divisors, define \emph{Weil's product} as
        \[
            \WP(f, g, S) \eqdef \frac{\displaystyle \prod_{z \in S} f(z)^{\ord_{z}{g}}}{\displaystyle \prod_{z \in S} g(z)^{\ord_{z}{f}}}. 
        \]
    \end{definition}  

    \begin{remark}
        Weil's reciprocity law is equivalent to $\WP(f, g, S) = 1$ for a compact Riemann surface $S$ without boundary.  
    \end{remark}  

    So, one can try to ``combinatorially'' prove Weil's reciprocity law by cutting the Riemann surface $S$ into cylinders, computing Weil's product separately for each cylinder, and then gluing everything back together. 

    In the following lemma, we show that Weil's product for a cylinder may not be equal to 1 (since it is not closed), but the correction looks quite nice. 

    \begin{lemma}\label{lemma_1} 
      Let $f, g$ be meromorphic functions on a cylinder $C = \{ z \in \mathbb{C}\ \vert \ R_{1} < z < R_{2} \}$ with  disjoint divisors $(f)$ and $(g)$ with supports $\supp(f) = \{ a_1, \ldots, a_n \}$ and $\supp(g) = \{ b_1, \ldots, b_m \}$. Let $S_2$ and $S_1$ be the cylinder's external and internal boundary circles $C$. Then 
      \begin{equation}
        \WP(f, g,C) = \frac{\varphi(S_2)}{\varphi(S_1)},
      \end{equation}
      where 
      \begin{equation}
        \varphi(S_i) = \varphi(f, g, S_i) = \exp\left(\frac{1}{2\pi i} \int\limits_{S_i} \log{f} \frac{dg}{g}\right).
      \end{equation}

    \end{lemma}

    \begin{proof}
      First note that 
        \begin{equation}
            \WP(f, g,C)= \exp\left( \sum_{j = 1}^{m} \log{f(b_j)} \cdot \ord_{b_j}{g} - \sum_{i = 1}^{n} \log{g(a_i)} \cdot \ord_{a_i}{f} \right),
        \end{equation}
        so we need to compute the right-hand side. Using the residue theorem, we obtain that  
        \begin{equation}
            \int\limits_{ \partial \left( \bigcup B_{\varepsilon}(b_j) \right) } \log{f} \frac{dg}{g} = 2\pi i \sum_{j = 1}^{m} \log{f(b_j)} \cdot \ord_{b_j}{g} = B,
              \int\limits_{ \partial \left( \bigcup B_{\varepsilon}(a_i) \right) } \log{g} \frac{df}{f} = 2\pi i \sum_{i = 1}^{n} \log{g(a_i)} \cdot \ord_{a_i}{g} = A \label{eq_13}
        \end{equation}
        
         Consider the manifold $C_\varepsilon = C - \bigcup B_{\varepsilon}(a_i) - \bigcup B_{\varepsilon}(b_j)$. As the form $\frac{df}{f} \wedge \frac{dg}{g}$ is non-singular on $C_\varepsilon$, we have $\int_{C_\varepsilon} \frac{df}{f} \wedge \frac{dg}{g} = 0$. So, from the Stokes formula, we get 
         \begin{equation}
             1 = \exp\left( \frac{1}{2\pi i} \int\limits_{C_{\varepsilon}} \frac{df}{f} \wedge \frac{dg}{g} \right) = \frac{\displaystyle   \exp\left( \frac{1}{2\pi i}  \int\limits_{\partial C} \log{f} \frac{dg}{g} \right) }{\displaystyle  \exp\left( \frac{1}{2\pi i} \cdot (B - A) \right) } \label{eq_14}
         \end{equation}
         
     \end{proof}

         Since $\log$ is multivalued, we would instead write the exponents of integrals, and thus, the numbers above are well-defined.  From the Lemma~\ref{lemma_1}, we see that Weil's product over a cylinder may not be equal to 1, but the correction that comes out depends only on the integral of some meromorphic form over the boundary circles. 

    We are ready to provide an alternative proof of Weil's reciprocity.

    \begin{proof}[Alternative proof of the theorem~\ref{complex_weil_reciprocity}]
    Let us cut the Riemann surface $S$ into pair-of-pants and then cut every pair-of-pants into three cylinders (see Figure~\ref{fig:7}). Since the surface is closed, when we glue all the pieces, factors coming from the boundary circles telescopically cancel, and we obtain that Weil's product over $S$ equals $1$.  For the case of a sphere (i.e., a surface of genus $g = 0$), we need to cut it into two disks and a cylinder, and the same cancelation happens.

      \begin{figure}[htbp]
        \centering
         \begin{tikzpicture}[scale = 0.51]
          \node at (0,0){\includegraphics[scale = 0.75]{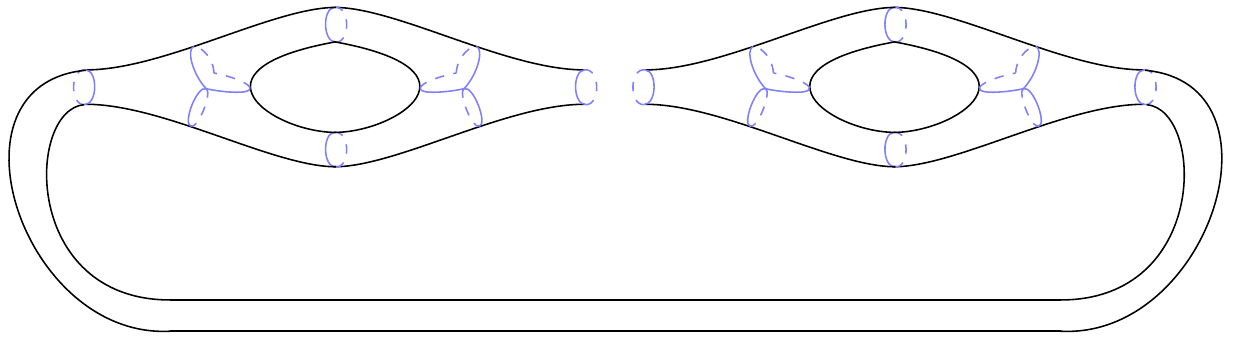} };
           \node[right] at (-0.58, 2.1){\tiny \( \dots \)};
        \end{tikzpicture}
        \caption{Cutting the Riemann surface $S$ into the cylinders.}
        \label{fig:7}
        \end{figure}
    \end{proof}
    
{\bf Beilinson functionals.} Let $\alpha$ be a point and $\gamma$ be loop in a Riemann surface $C$, which passes through $\alpha$. Define Beilinson functionals (\cite{weil1} and~\cite{beilinson1980higher})
$$c(f,g)(\gamma) = \left(\frac{1}{2\pi i}\right)^2\left(\int_\gamma \log f d \log g - \log g (\alpha)\int_\gamma d \log f)\right)\in \mathbb{C}/\mathbb{Z},$$ 

$$[f,g,\gamma]= \exp[2\pi i \cdot  c(f,g)(\gamma)].$$

      Our integral $\varphi(S)$ from  Lemma~\ref{lemma_1}  is similar to Beilinson's integral. If $\gamma$ is a small loop around $p$ then $[f,g]_p = [f,g,\gamma]$.

 %    At the same time, whereas in Beilinson's proof the idea is that  Beilinson integral $B_{\gamma}(f, g)$ depends only on the homological class of the cycle $\gamma$, in our proof the idea is that the integral $\varphi$ depends only of the boundary circles of the surface for a fixed pair of functions. 

\section{Tropical Weil pairing }

Let $f$ be a tropical meromorphic function on a tropical curve $\Gamma$. The formal sum $(f)=\sum_{p\in\Gamma} \ord_pf \cdot p$ is called the divisor of $f$. The sum of the coefficients of a divisor $D$ is called the degree of $D$. Note that the degree of $(f)$ is zero since $\sum_{p\in\Gamma} \ord_p{f} = 0$. However, given a divisor $D$ of degree zero, it is not always possible to find a tropical meromorphic function $f$ such that $D = (f)$, but if such a function exists, it is unique up to an additive constant.

Consider continuous piece-wise linear functions on $\Gamma$ with no restrictions on slopes. Such functions are called tropical quasi-meromorphic functions. The notion of the order $\ord_p f$ and the divisor  $\sum p\cdot \ord_p f $ of such a function $f$ is defined as for the tropical meromorphic function. But now, given any divisor (with real coefficients) $D$ of degree zero, one can always find a unique (up to an additive constant) tropical quasi-meromorphic function $f$ such that $D=(f)$.

\begin{theorem} 
  Let $D = \sum_{i} q_i \cdot p_i$  be a divisor (on a tropical curve $\Gamma$) of degree zero with real coefficients $q_i$. Then, there exists a unique (up to an additive constant) tropical quasi-meromorphic function $f$ such that $D = (f)$. 
\end{theorem}
\begin{proof} Uniqueness follows from the fact that a function $f$ with $\ord_p f=0$ for all $p$ must be equal everywhere to its maximal value.
  
    Let us consider $\Gamma$ as an electric circuit: at each point $p_i \in \supp{D}$ place an electric charge $q_i$ and let the resistance on a segment $[a, b] \subset \Gamma$ be equal to the length of the segment. The total charge equals the degree of $D$, i.e., zero. This electric circuit uniquely defines currents on all edges.

    Then, construct a continuous piece-wise function $f$ with a slope at $p \in \Gamma$ equal to the electric current at $p$.  The function $f$ represents the electric potential.  Kirchhoff's first law gives us that for every $p_i \in \supp{D}$ we have $\sum I_j = \ord_{p_i}{f}$ where $I_j$ are electric currents around $p$, so $\ord_{p_i}{f} = q_i$ and $(f) = D$.  
    
    In the hydraulic analogy (water in pipes, no gravitation, pipes correspond to tropical edges), the function $f$ represents the pressure, and the lengths of the edges of the tropical curve correspond to the areas of the sections of pipes (note that in the hydraulic analogy, the lengths of pipes are irrelevant). The flow rate (liters per second) corresponds to the current in this analogy. 
\end{proof}

\begin{definition} Let $\Gamma$ be a tropical curve and $D_1=\sum_{i=1}^n a_ip_i,D_2=\sum_{j = 1}^m b_jq_j$ be two divisors of degree zero on $\Gamma$. The tropical Weil pairing $TW(D_1, D_2)$ is defined as follows. Find tropical harmonic functions $f_1,f_2$ such that $(f_1)=D_1, (f_2)=D_2$. Then

$$TW(D_1,D_2) = \sum _{i=1}^n a_i f_2(p_i)= \sum_{j=1}^mb_jf_1(q_j).$$

\end{definition}

Note that the second equality in the definition follows from the tropical Weil reciprocity law, so the tropical Weil pairing is well-defined.

What could be the classical version of the tropical Weil pairing? Following G. Mikhalkin's suggestion, we can do the following:

Given a divisor $D=\sum_{i=1}^na_np_n$ of degree zero on a Riemann surface $S$, we can define a unique real-normalized meromorphic differential $\omega$  whose all periods are real and the residue at $p_i$ is equal to $a_i$ for $i=1,\dots,n$. Such differentials were introduced in \cite{grushevsky2009universal} and then used in \cite{krichever2012real}.
Choose any point $z_0\in S$. Then, the real part of the function $f(z)=\frac{1}{i}\int_{z_0}^z\omega$ does not depend on the path of the integration and so is well defined. 

The classical version of the tropical Weil's pairing between $D$ and a divisor $D'$ of degree zero will be the sum of the values of $\mathrm{Re}{f}$ over $D'$.  This also corresponds to the adelic pairing on differentials of the second kind as in \cite{takhtajan2013quantum}, p. 387.  

It would be interesting to study the tropical Weil pairing for divisors with restrictions on the slopes of the corresponding tropical quasi-meromorphic functions. Divisors of finite order correspond to rational slopes. However, we can consider slopes from a given (real) number field.

% bibliography plain 
%\bibliographystyle{plain}
%\bibliography{../../bibliography.bib}

\end{document}